\documentclass{amsart}

\usepackage{graphicx}
\usepackage{amssymb}
\usepackage{amscd}
\usepackage{latexsym}
\usepackage{amsfonts}
\usepackage{ragged2e}
\usepackage{amsmath}
\usepackage{float}
\usepackage[latin1]{inputenc}
\usepackage[english]{babel}

\newtheorem{proposition}{Proposition}
\newtheorem{lemma}{Lemma}
\newtheorem{theorem}{Theorem}
\newtheorem{definition}{Definition}
\newtheorem{remark}{Remark}
\newtheorem{conjecture}{Conjecture}

\newcommand\zkn[1][k]{Z_{#1}^{(n)}}
\newcommand\zk[1][k]{Z_{#1}^{(2)}}
\DeclareMathOperator{\SU}{SU} \DeclareMathOperator\sll{sl}
\newcommand{\im}{\mathop{\fam0 Im}\nolimits}

\newcommand{\tr}{\mathop{\fam0 Tr}\nolimits}

\newcommand{\rank}{\mathop{\fam0 Rank}\nolimits}

\newcommand{\Hom}{\mathop{\fam0 Hom}\nolimits}

\newcommand{\End}{\mathop{\fam0 End}\nolimits}

\newcommand{\id}{\mathop{\fam0 Id}\nolimits}

\newcommand{\Aut}{\mathop{\fam0 Aut}\nolimits}

\newcommand{\Id}{\mathop{\fam0 Id}\nolimits}

\newcommand{\bC}{{\mathbb C}}
\newcommand{\bR}{{\mathbb R}}
\newcommand{\bT}{{\bar T}}
\newcommand{\C}{C}

\newcommand{\Z}{{\mathbb Z}}
\newcommand{\bZ}{\Z{}}

\newcommand{\D}{{\mathcal D}}
\newcommand{\ra}{\mathop{\fam0 \rightarrow}\nolimits}

\newcommand{\dbar}{ \bar \partial}

\newcommand{\cH}{ {\mathcal H}}

\newcommand{\T}{ {\mathcal T}}

\newcommand{\V}{ {\mathcal V}}
\renewcommand{\L}{{\mathcal L}}

\renewcommand{\P}{ {\mathbb P}}

\newcommand{\tD}{ {\tilde D}}
\newcommand{\tG}{ {\tilde G}}
\newcommand{\tW}{ {\tilde W}}
\newcommand{\s}{\sigma}

\newcommand{\Nabla}{{\mathbf {\hat \nabla}}}
\newcommand{\Nablat}{{\mathbf {\hat \nabla}}^{t}}
\newcommand{\Nablae}{{\mathbf {\hat \nabla}}^e}
\newcommand{\Nablaet}{{\mathbf {\hat \nabla}}^{e,t}}
\newcommand{\BTstar}{\star^{\text{\tiny BT}}}
\newcommand{\tBTstar}{{\tilde \star}^{\text{\tiny BT}}}

\newcommand{\Teim}{Teichm{\"u}ller }
\newcommand{\Ric}{\mathop{\fam0 Ric}\nolimits}

\begin{document}

\title[Hitchin's connection, Toeplitz Operators and curve operators]
{Hitchin's projectively flat connection, Toeplitz operators and the
  asymptotic expansion of TQFT curve operators}

\author{J{\o}rgen Ellegaard Andersen} \address{Center for the Topology
  and Quantization of Moduli
  Spaces\\ Department of Mathematics\\
  University of Aarhus\\
  DK-8000, Denmark}

\email{andersen@imf.au.dk}

\author{Niels Leth Gammelgaard} \address{Center for the Topology and
  Quantization of Moduli
  Spaces\\ Department of Mathematics\\
  University of Aarhus\\
  DK-8000, Denmark}

\email{nlg@imf.au.dk}

\maketitle

\begin{abstract}
  In this paper, we will provide a review of the geometric
  construction, proposed by Witten, of the $\SU(n)$ quantum
  representations of the mapping class groups which are part of the
  Reshetikhin-Turaev TQFT for the quantum group $U_q(sl(n,\bC))$. In
  particular, we recall the differential geometric construction of
  Hitchin's projectively flat connection in the bundle over
  Teichmüller space obtained by push-forward of the determinant line
  bundle over the moduli space of rank $n$, fixed determinant,
  semi-stable bundles fibering over Teichmüller space. We recall the
  relation between the Hitchin connection and Toeplitz operators which
  was first used by the first named author to prove the asymptotic
  faithfulness of the $\SU(n)$ quantum representations of the mapping
  class groups.  We further review the construction of the formal
  Hitchin connection, and we discuss its relation to the full
  asymptotic expansion of the curve operators of Topological Quantum
  Field Theory. We then go on to identifying the first terms in the
  formal parallel transport of the Hitchin connection explicitly. This
  allows us to identify the first terms in the resulting star product
  on functions on the moduli space. This is seen to agree with the
  first term in the star-product on holonomy functions on these moduli
  spaces defined by Andersen, Mattes and Reshetikhin.
\end{abstract}

\section{Introduction}

Witten constructed, via path integral techniques, a quantization of
Chern-Simons theory in $2+1$ dimensions, and he argued in \cite{W1}
that this produced a TQFT, indexed by a compact simple Lie group and
an integer level $k$. For the group $\SU(n)$ and level $k$, let us
denote this TQFT by $\zkn$. Witten argues in \cite{W1} that the theory
$\zk$ determines the Jones polynomial of a knot in $S^3$.
Combinatorially, this theory was first constructed by Reshetikhin and
Turaev, using representation theory of $U_q(\sll(n,\C))$ at
$q=e^{(2\pi i)/(k+n)}$, in \cite{RT1} and \cite{RT2}. Subsequently,
the TQFT's $\zkn$ were constructed using skein theory by Blanchet,
Habegger, Masbaum and Vogel in \cite{BHMV1}, \cite{BHMV2} and
\cite{B1}.

The two-dimensional part of the TQFT $\zkn$ is a modular functor with
a certain label set. For this TQFT, the label set $\Lambda_k^{(n)}$ is
a finite subset (depending on $k$) of the set of finite dimensional
irreducible representations of $\SU(n)$. We use the usual labeling of
irreducible representations by Young diagrams, so in particular
$\Box\in \Lambda_k^{(n)}$ is the defining representation of
$\SU(n)$. Let further $\lambda_0^{(d)} \in \Lambda_k^{(n)}$ be the
Young diagram consisting of $d$ columns of length $k$. The label set
is also equipped with an involution, which is simply induced by taking
the dual representation. The trivial representation is a special
element in the label set which is clearly preserved by the involution.

\begin{align*}
  \zkn: \quad \left\{\parbox{3cm}{\RaggedRight Category of (extended)
      closed oriented surfaces with $\Lambda^{(n)}_k$-labeled marked
      points with projective tangent vectors} \right\} \ \to \
  \left\{\parbox{3cm}{\RaggedRight Category of finite dimensional
      vector spaces over $\bC$} \right\}
\end{align*}

The three-dimensional part of $\zkn$ is an association of a vector,
$$\zkn(M,L,\lambda)\in \zkn(\partial M, \partial L, \partial \lambda),$$
to any compact, oriented, framed $3$--manifold $M$ together with an
oriented, framed link $(L,\partial L)\subseteq (M,\partial M)$ and a
$\Lambda_k^{(n)}$-labeling $\lambda : \pi_0(L) \ra \Lambda_k^{(n)}$.

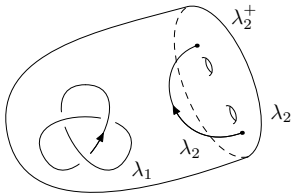
\begin{figure}[H]
  \centering
  \includegraphics[scale=0.8]{slide-figs.2}
\end{figure}

This association has to satisfy the Atiyah-Segal-Witten TQFT axioms
(see e.g. \cite{At}, \cite{Segal} and \cite{W1}). For a more
comprehensive presentation of the axioms, see Turaev's book \cite{T}.

The geometric construction of these TQFTs was proposed by Witten in
\cite{W1} where he derived, via the Hamiltonian approach to quantum
Chern-Simons theory, that the geometric quantization of the moduli
spaces of flat connections should give the two-dimensional part of the
theory. Further, he proposed an alternative construction of the
two-dimensional part of the theory via WZW-conformal field
theory. This theory has been studied intensively. In particular, the
work of Tsuchiya, Ueno and Yamada in \cite{TUY} provided the major
geometric constructions and results needed. In \cite{BK}, their
results were used to show that the category of integrable highest
weight modules of level $k$ for the affine Lie algebra associated to
any simple Lie algebra is a modular tensor category. Further, in
\cite{BK}, this result is combined with the work of Kazhdan and
Lusztig \cite{KL} and the work of Finkelberg \cite{Fi} to argue that
this category is isomorphic to the modular tensor category associated
to the corresponding quantum group, from which Reshetikhin and Turaev
constructed their TQFT. Unfortunately, these results do not allow one
to conclude the validity of the geometric constructions of the
two-dimensional part of the TQFT proposed by Witten.  However, in
joint work with Ueno, \cite{AU1}, \cite{AU2}, \cite{AU3} and
\cite{AU4}, we have given a proof, based mainly on the results of
\cite{TUY}, that the TUY-construction of the WZW-conformal field
theory, after twist by a fractional power of an abelian theory,
satisfies all the axioms of a modular functor.  Furthermore, we have
proved that the full $2+1$-dimensional TQFT resulting from this is
isomorphic to the aforementioned one, constructed by BHMV via skein
theory. Combining this with the theorem of Laszlo \cite{La1}, which
identifies (projectively) the representations of the mapping class
groups obtained from the geometric quantization of the moduli space of
flat connections with the ones obtained from the TUY-constructions,
one gets a proof of the validity of the construction proposed by
Witten in \cite{W1}.

Part of this TQFT is the quantum $\SU(n)$ representations of the
mapping class groups. Namely, if $\Sigma$ is a closed oriented
surfaces of genus $g$, $\Gamma$ is the mapping class group of
$\Sigma$, and $p$ is a point on $\Sigma$, then the modular functor
induces a representation
\begin{equation}\label{rep}
  Z^{(n,d)}_{k} : \Gamma \to \P\Aut\bigl(\zkn (\Sigma, p, \lambda_0^{(d)})\bigr).
\end{equation}
For a general label of $p$, we would need to choose a projective
tangent vector $v_p\in T_p\Sigma/\bR_+$, and we would get a
representation of the mapping class group of $(\Sigma,p, v_p)$.  But
for the special labels $\lambda_0^{(d)}$, the dependence on $v_p$ is
trivial and in fact we get a representation of $\Gamma$.  Furthermore,
the curve operators are also part of any TQFT: For
$\gamma\subseteq\Sigma-\{p\}$ an oriented simple closed curve and any
$\lambda\in\Lambda_k^{(n)}$, we have the operators
\begin{equation}\label{co}
  Z^{(n,d)}_{k}  (\gamma,\lambda): \zkn (\Sigma, p, \lambda_0^{(d)}) \to
  \zkn (\Sigma, p, \lambda_0^{(d)}),
\end{equation}
defined as
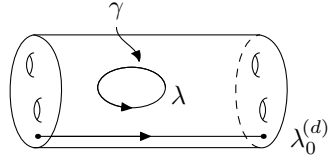
\[Z^{(n,d)}_{k} (\gamma,\lambda) = Z^{(n,d)}_{k} (\Sigma \times I,
\gamma\times \{\frac12\} \coprod \{p\}\times I,
\{\lambda,\lambda^{(d)}_0\}). \]
\begin{figure}[H]
  \centering
  \includegraphics{slide-figsv2.1}
\end{figure}
The curve operators are natural under the action of the mapping class
group, meaning that following diagram,

\begin{equation*}
  \begin{CD}
    \zkn (\Sigma, p, \lambda_0^{(d)}) @>Z^{(n,d)}_{k}
    (\gamma,\lambda)>> \zkn (\Sigma, p, \lambda_0^{(d)}) \\ @V
    Z^{(n,d)}_{k}(\phi) VV @VV Z^{(n,d)}_{k}(\phi) V\\
    \zkn (\Sigma, p, \lambda_0^{(d)}) @>Z^{(n,d)}_{k}
    (\phi(\gamma),\lambda)>> \zkn (\Sigma, p, \lambda_0^{(d)}),
  \end{CD}
\end{equation*}
is commutative for all $\phi\in\Gamma$ and all labeled simple closed
curves $(\gamma,\lambda)\subset \Sigma-\{p\}$.

For the curve operators, we can derive an explicit formula using
factorization: Let $\Sigma'$ be the surface obtained from cutting
$\Sigma$ along $\gamma$ and identifying the two boundary components to
two points, say $\{p_+,p_-\}$. Here $p_+$ is the point corresponding
to the "left" side of $\gamma$. For any label $\mu\in
\Lambda_k^{(n)}$, we get a labeling of the ordered points $(p_+,p_-)$
by the ordered pair of labels $(\mu,\mu^\dagger)$.

Since $\zkn$ is also a modular functor, one can factor the space
$\zkn(\Sigma, p, \lambda_0^{(d)})$ as a direct sum, 'along' $\gamma$,
over $\Lambda_k^{(n)}$. That is, we get an isomorphism

\begin{equation}
  \zkn(\Sigma, p, \lambda_0^{(d)}) \cong \bigoplus_{\mu\in
    \Lambda_k^{(n)}} Z^{(k)}(\Sigma',
  p_+,p_-,p,\mu,\mu^\dagger,\lambda_0^{(d)} ).\label{Fact}
\end{equation}
Strictly speaking, we need a base point on $\gamma$ to induce tangent
directions at $p_\pm$. However, the corresponding subspaces of
$Z^{(k)}(\Sigma, p, \lambda_0^{(d)})$ do not depend on the choice of
base point.  The isomorphism (\ref{Fact}) induces an isomorphism
\begin{equation*}
  \End(Z^{(k)}(\Sigma, p, \lambda_0^{(d)})) \cong \bigoplus_{\mu\in
    \Lambda_k^{(n)}} \End(Z^{(k)}(\Sigma',
  p_+,p_-,p,\mu,\mu^\dagger,\lambda_0^{(d)})),
\end{equation*}
which also induces a direct sum decomposition of $\End(Z^{(k)}(\Sigma,
p, \lambda_0^{(d)}))$, independent of the base point.

The TQFT axioms imply that the curve operator
$Z^{(k)}(\gamma,\lambda)$ is diagonal with respect to this direct sum
decomposition along $\gamma$. One has the formula
\[Z^{(k)}(\gamma,\lambda) = \bigoplus_{\mu\in \Lambda_k^{(n)}}
S_{\lambda,\mu} (S_{0,\mu})^{-1}\Id_{Z^{(k)}(\Sigma',
  p_+,p_-,p,\mu,\mu^\dagger,\lambda_0^{(d)})}.\] Here $
S_{\lambda,\mu}$ is the $S$-matrix\footnote{The $S$-matrix is
  determined by the isomorphism that a modular functor induces from
  two different ways of glueing an annalus to obtain a torus. For its
  definition, see e.g. \cite{MS}, \cite{Segal}, \cite{Walker} or
  \cite{BK} and references in there. It is also discussed in
  \cite{AU3}.} of the theory $\zkn$. See e.g. \cite{B1} for a
derivation of this formula.

Let us now briefly recall the geometric construction of the
representations $Z^{(n,d)}_k$ of the mapping class group, as proposed
by Witten, using geometric quantization of moduli spaces.

We assume from now on that the genus of the closed oriented surface
$\Sigma$ is at least two. Let $M$ be the moduli space of flat $\SU(n)$
connections on $\Sigma - p$ with holonomy around $p$ equal to
$\exp(2\pi i d/n)\Id \in \SU(n)$. When $(n,d)$ are coprime, the moduli
space is smooth. In all cases, the smooth part of the moduli space has
a natural symplectic structure $\omega$.  There is a natural smooth
symplectic action of the mapping class group $\Gamma$ of $\Sigma$ on
$M$. Moreover, there is a unique prequantum line bundle $(\L,\nabla,
(\cdot,\cdot))$ over $(M,\omega)$. The Teichm\"{u}ller space $\T$ of
complex structures on $\Sigma$ naturally, and $\Gamma$-equivariantly,
parametrizes K\"{a}hler structures on $(M,\omega)$. For $\sigma\in \T
$, we denote by $M_\sigma$ the manifold $(M,\omega)$ with its
corresponding K\"{a}hler structure.  The complex structure on
$M_\sigma$ and the connection $\nabla$ in $\L$ induce the structure of
a holomorphic line bundle on $\L$. This holomorphic line bundle is
simply the determinant line bundle over the moduli space, and it is an
ample generator of the Picard group \cite{DN}.

By applying geometric quantization to the moduli space $M$, one gets,
for any positive integer $k$, a certain finite rank bundle over \Teim
space $\T$ which we will call the {\em Verlinde bundle} $\V_k$ at
level $k$. The fiber of this bundle over a point $\sigma\in \T$ is
$\V_{k,\sigma} = H^0(M_\sigma,\L^k)$. We observe that there is a
natural Hermitian structure $\langle\cdot,\cdot\rangle$ on
$H^0(M_\sigma,\L^k)$ by restricting the $L_2$-inner product on global
$L_2$ sections of $\L^k$ to $H^0(M_\sigma,\L^k)$.

The main result pertaining to this bundle is:
\begin{theorem}[Axelrod, Della Pietra and Witten;
  Hitchin]\label{projflat}
  The projectivization of the bundle $\V_k$ supports a natural flat
  $\Gamma$-invariant connection $\Nabla$.
\end{theorem}

This is a result proved independently by Axelrod, Della Pietra and
Witten \cite{ADW} and by Hitchin \cite{H}. In section \ref{ghc}, we
review our differential geometric construction of the connection
$\Nabla$ in the general setting discussed in \cite{A5}. We obtain as a
corollary that the connection constructed by Axelrod, Della Pietra and
Witten projectively agrees with Hitchin's.

\begin{definition}\label{MD1}
  We denote by $Z^{(n,d)}_{k}$ the representation,
  \begin{align*}
    Z^{(n,d)}_{k} : \Gamma \to \P\Aut\bigl(\zkn (\Sigma, p,
    \lambda_0^{(d)})\bigr),
  \end{align*}
  obtained from the action of the mapping class group on the covariant
  constant sections of $\P(\V_k)$ over $\T$.
\end{definition}

The projectively flat connection $\Nabla$ induces a {\em flat}
connection $\Nablae$ in $\End(\V_k)$. Let $\End_0(\V_k)$ be the
subbundle consisting of traceless endomorphisms. The connection
$\Nablae$ also induces a connection in $\End_0(\V_k)$, which is
invariant under the action of $\Gamma$.

In \cite{A3}, we proved

\begin{theorem}[Andersen]\label{MainA3} Assume that $n$ and $d$ are coprime or
  that $(n,d)=(2,0)$ when $g=2$.  Then, we have that
  \begin{equation*}
    \bigcap_{k=1}^\infty \ker(Z^{(n,d)}_k) =
    \begin{cases}
      \{1, H\} & g=2 \mbox{, }n=2 \mbox{ and } d=0 \\ \{1\}&
      \mbox{otherwise},
    \end{cases}
  \end{equation*}
  where $H$ is the hyperelliptic involution.
\end{theorem}

The main ingredient in the proof of this Theorem is the {\em Toeplitz
  operators} associated to smooth functions on $M$. For each $f\in
C^\infty(M)$ and each point $\sigma\in \T$, we have the Toeplitz
operator,
\[T^{(k)}_{f,\sigma} : H^0(M_{\sigma},\L_\s^k) \ra
H^0(M_{\sigma},\L_\s^k),\] which is given by
\[T^{(k)}_{f,\sigma} = \pi^{(k)}_\sigma(fs)\] for all $s\in
H^0(M_{\sigma},\L_\s^k)$. Here $\pi^{(k)}_\sigma$ is the orthogonal
projection onto $H^0(M_{\sigma},\L_\s^k)$ induced from the $L_2$-inner
product on $\C^\infty(M,\L^k)$. We get a smooth section of
$\End(\V^{(k)})$,
\[T_{f}^{(k)} \in C^\infty(\T,\End(\V^{(k)})), \] by letting
$T_{f}^{(k)}(\sigma) = T_{f,\sigma}^{(k)}$ (see \cite{A3}). See
section \ref{BZdq} for further discussion of the Toeplitz operators
and their connection to deformation quantization.

The sections $T_{f}^{(k)}$ of $\End(\V^{(k)})$ over $\T$ are not
covariant constant with respect to Hitchin's connection $\Nablae$.
However, they are asymptotically as $k$ goes to infinity. This will be
made precise when we discuss the formal Hitchin connection below.

As a further application of TQFT and the theory of Toeplitz operators
together with the theory of coherent states, we recall the first
authors solution to a problem in geometric group theory, which has
been around for quite some time (see e.g. Problem (7.2) in Chapter 7,
"A short list of open questions", of \cite{BHV}): In \cite{A7},
Andersen proved that

\begin{theorem}[Andersen]\label{Main}
  The mapping class group of a closed oriented surface, of genus at
  least two, does not have Kazhdan's property (T).
\end{theorem}

Returning to the geometric construction of the Reshetikhin-Turaev
TQFT, let us recall the geometric construction of the curve operators.
First of all, the decomposition (\ref{Fact}) is geometrically obtained
as follows (see \cite{A6} for the details):

One considers a one parameter family of complex structures
$\sigma_t\in \T$, $t\in \bR_+$, such that the corresponding family in
the moduli space of curves converges in the Mumford-Deligne boundary
to a nodal curve, which topologically corresponds to shrinking
$\gamma$ to a point. By the results of \cite{A1.5}, the corresponding
sequence of complex structures on the moduli space $M$ converges to a
non-negative polarization on $M$ whose isotropic foliation is spanned
by the Hamiltonian vector fields associated to the holonomy functions
of $\gamma$. The main result of \cite{A6} is that the covariant
constant sections of $\V^{(k)}$ along the family $\sigma_t$ converges
to distributions supported on the Bohr-Sommerfeld leaves of the
limiting non-negative polarization as $t$ goes to infinity. The direct
sum of the geometric quantization of the level $k$ Bohr-Sommerfeld
levels of this non-negative polarization is precisely the left-hand
side of (\ref{Fact}). A sewing-construction, inspired by conformal
field theory (see \cite{TUY}), is then applied to show that the
resulting linear map from the right-hand side of (\ref{Fact}) to the
left-hand side is an isomorphism. This is described in detail in
\cite{A6}.

In \cite{A6}, we further prove the following important asymptotic
result. Let $h_{\gamma,\lambda}\in C^\infty(M)$ be the holonomy
function obtained by taking the trace in the representation $\lambda$
of the holonomy around $\gamma$.

\begin{theorem}[Andersen]\label{MainA6}
  For any one-dimensional oriented submanifold $\gamma$ and any
  labeling $\lambda$ of the components of $\gamma$, we have that $$
  \lim_{k \ra \infty} \|Z^{(n,d)}_k(\gamma, \lambda) -
  T^{(k)}_{h_{\gamma,\lambda}}\| = 0. $$
\end{theorem}

Let us here give the main idea behind the proof of Theorem
\ref{MainA6} and refer to \cite{A6} for the details. One considers the
explicit expression for the $S$-matrix, as given in formula (13.8.9)
in Kac's book \cite{Kac}
\begin{equation}
  S_{\lambda,\mu}/S_{0,\mu} = \lambda(e^{-2 \pi i \frac{\check{\mu}
      + \check{\rho}}{k+n}}),\label{KacS}
\end{equation}
where $\rho$ is half the sum of the positive roots and $\check{\nu}$
($\nu$ any element of $\Lambda$) is the unique element of the Cartan
subalgebra of the Lie algebra of $\SU(n)$ which is dual to $\nu$ with
respect to the Cartan-Killing form $(\cdot,\cdot)$.

From the expression (\ref{KacS}), one sees that under the isomorphism
$\check{\mu}\mapsto \mu$, the expression $S_{\lambda,\mu}/S_{0,\mu}$
makes sense for any $\check{\mu}$ in the Cartan subalgebra of the Lie
algebra of $\SU(n)$. Furthermore, one finds that the values of this
sequence of functions (depending on $k$) is asymptotic to the values
of the holonomy function $h_{\gamma,\lambda}$ at the level $k$
Bohr-Sommerfeld sets of the limiting non-negative polarizations
discussed above (see \cite{A1.5}). From this, one can deduce Theorem
\ref{MainA6}. See again \cite{A6} for details.

Let us now consider the general setting treated in \cite{A5}. Thus, we
consider, as opposed to only considering the moduli spaces, a general
prequantizable symplectic manifold $(M, \omega)$ with a prequantum
line bundle $(L,(\cdot,\cdot), \nabla)$. We assume that $\T$ is a
complex manifold which holomorphically and rigidly (see Definition
\ref{Ridig}) parameterizes K\"{a}hler structures on
$(M,\omega)$. Then, the following theorem, proved in \cite{A5},
establishes the existence of the Hitchin connection under a mild
cohomological condition.

\begin{theorem}[Andersen]\label{MainGHCI}
  Suppose that $I$ is a rigid family of K\"{a}hler structures on the
  compact, prequantizable symplectic manifold $(M,\omega)$ which
  satisfies that there exists an $n\in \bZ$ such that the first Chern
  class of $(M,\omega)$ is $n [\frac{\omega}{2\pi}]\in H^2(M,\bZ)$ and
  \mbox{$H^1(M,\bR) = 0$}. Then, the Hitchin connection $\Nabla$ in
  the trivial bundle $\cH^{(k)} = \mathcal{T} \times C^\infty(M,
  \mathcal{L}^k)$ preserves the subbundle $H^{(k)}$ with fibers
  $H^0(M_\sigma, \mathcal{L}^k)$. It is given by
  \[\Nabla_V = \Nablat_V + \frac1{4k+2n} \left\{\Delta_{G(V)} +
    2\nabla_{G(V)dF} + 4k V'[F]\right\},\] where $\Nablat$ is the
  trivial connection in $\cH^{(k)}$, and $V$ is any smooth vector
  field on $\T$.
\end{theorem}

In section \ref{fgHc}, we study the formal Hitchin connection which
was introduced in \cite{A5}. Let $\D(M)$ be the space of smooth
differential operators on $M$ acting on smooth functions on $M$. Let
$\C_h$ be the trivial $C^\infty_h(M)$-bundle over $\T$.

\begin{definition}\label{fc2}
  A formal connection $D$ is a connection in $\C_h$ over $\T$ of the
  form
  \[D_V f = V[f] + \tD(V)(f),\] where $\tD$ is a smooth one-form on
  $\T$ with values in $\D_h(M) = \D(M)[[h]]$, $f$ is any smooth
  section of $\C_h$, $V$ is any smooth vector field on $\T$ and $V[f]$
  is the derivative of $f$ in the direction of $V$.
\end{definition}

Thus, a formal connection is given by a formal series of differential
operators
\[\tD(V) = \sum_{l=0}^\infty \tD^{(l)}(V) h^l.\]

From Hitchin's connection in $H^{(k)}$, we get an induced connection
$\Nablae$ in the endomorphism bundle $\End(H^{(k)})$. As previously
mentioned, the Teoplitz operators are not covariant constant sections
with respect to $\Nablae$, but asymptotically in $k$ they are. This
follows from the properties of the formal Hitchin connection, which is
the formal connection $D$ defined through the following theorem
(proved in \cite{A5}).

\begin{theorem}(Andersen)\label{MainFGHCI2} There is a unique formal
  connection $D$ which satisfies that
  \begin{equation}
    \Nablae_V T^{(k)}_f \sim T^{(k)}_{(D_V f)(1/(k+n/2))}\label{Tdf2}
  \end{equation}
  for all smooth section $f$ of $\C_h$ and all smooth vector fields on
  $\T$. Moreover,
  \[\tD = 0 \mod h.\]
  Here $\sim$ means the following: For all $L\in \Z_+$ we have that
  \[\left\| \Nablae_V T^{(k)}_{f} - \left( T^{(k)}_{V[f]} +
      \sum_{l=1}^L T_{\tD^{(l)}_V f}^{(k)} \frac1{(k+n/2)^{l}}\right)
  \right\| = O(k^{-(L+1)}),\] uniformly over compact subsets of $\T$,
  for all smooth maps $f:\T \ra C^\infty(M)$.
\end{theorem}

Now fix an $f\in C^\infty(M)$, which does not depend on $\sigma \in
\mathcal{T}$, and notice how the fact that $\tilde D = 0 \mbox{ mod }
h$ implies that
\[\left\| \Nablae_V T^{(k)}_{f} \right\| = O(k^{-1}).\]
This expresses the fact that the Toeplitz operators are asymptotically
flat with respect to the Hitchin connection.

We define a mapping class group equivariant formal trivialization of
$D$ as follows.

\begin{definition}\label{formaltrivi2}
  A formal trivialization of a formal connection $D$ is a smooth map
  $P : \T \ra \D_h(M)$ which modulo $h$ is the identity, for all
  $\s\in \T$, and which satisfies
  \[D_V(P(f)) = 0,\] for all vector fields $V$ on $\T$ and all $f\in
  C^\infty_h(M)$. Such a formal trivialization is mapping class group
  equivariant if $P(\phi(\sigma)) = \phi^* P(\sigma)$ for all $\sigma
  \in \T$ and $\phi\in \Gamma$.
\end{definition}

Since the only mapping class group invariant functions on the moduli
space are the constant ones (see \cite{Go1}), we see that in the case
where $M$ is the moduli space, such a $P$, if it exists, must be
unique up to multiplication by a formal constant.

Clearly if $D$ is not flat, such a formal trivialization cannot exist
even locally on $\T$. However, if $D$ is flat and its zero-order term
is just given by the trivial connection in $C_h$, then a local formal
trivialization exists, as proved in \cite{A5}.

Furthermore, it is proved in \cite{A5} that flatness of the formal
Hitchin connection is implied by projective flatness of the Hitchin
connection. As was proved by Hitchin in \cite{H}, and stated above in
Theorem \ref{projflat}, this is the case when $M$ is the moduli
space. Furthermore, the existence of a formal trivialization implies
the existence of unique (up to formal scale) mapping class group
equivariant formal trivialization, provided that $H^1_\Gamma(\T,D(M))
= 0$. The first steps towards proving that this cohomology group
vanishes have been taken in \cite{AV1,AV2, AV3, Vi}.  In this paper,
we prove that

\begin{theorem}\label{Expliformula}
  The mapping class group equivariant formal trivialization of the
  formal Hitchin connection exists to first order, and we have the
  following explicit formula for the first order term of $P$
$$P_\sigma^{(1)}(f) = \frac{1}{4} \Delta_\sigma(f) + i \nabla_{X''_F}(f),$$
where $X''_F$ denotes the (0,1)-part of the Hamiltonian vector field
for the Ricci potential.
\end{theorem}

For the proof of the theorem, see section \ref{fgHc}. We will make the
following conjecture.

\begin{conjecture}\label{MainConj}
  The mapping class group equivariant formal trivialization of the
  formal Hitchin connection exists, and for any one-dimensional
  oriented submanifold $\gamma$ and any labeling $\lambda$ of the
  components of $\gamma$, we have the full asymptotic expansion
  \[Z^{(n,d)}_k(\gamma, \lambda) \sim
  T^{(k)}_{P(h_{\gamma,\lambda})},\] which means that for all $L$ and
  all $\sigma\in\T$, we have that
$$
\|Z^{(n,d)}_k(\gamma, \lambda) - \sum_{l=0}^L T^{(k)}_{P_\sigma^{(l)}
  (h_{\gamma,\lambda})} \frac{1}{(k+n/2)^l}\| = O(k^{L+1}). $$
\end{conjecture}

It is very likely that the techniques used in \cite{A6} to prove
Theorem \ref{MainA6} can be used to prove this conjecture.

When we combine this conjecture with the asymptotics of the product of
two Toeplitz operators (see Theorem \ref{tKS1}), we get the full
asymptotic expansion of the product of two curve operators:

\[Z^{(n,d)}_k(\gamma_1, \lambda_1)Z^{(n,d)}_k(\gamma_2, \lambda_2)
\sim T^{(k)}_{P(h_{\gamma_1,\lambda_1}) \tilde \star^{BT}_\sigma
  P(h_{\gamma_2,\lambda_2})},\] where $\tilde \star^{BT}_\sigma$ is
very closely related to the Berezin-Toeplitz star product for the
Kähler manifold $(M_\sigma,\omega)$, as first defined in
\cite{BMS}. See section \ref{BZdq} for further details regarding this.

Suppose that we have a mapping class group equivariant formal
trivialization $P$ of the formal Hitchin connection $D$. We can then
define a new smooth family of star products parametrized by $\T$ as
follows:
\[f\star_\s g = P_\s^{-1}(P_\s(f) \tBTstar_\s P_\s(g))\] for all
$f,g\in C^\infty(M)$ and all $\s\in \T$. Using the fact that $P$ is a
trivialization, it is not hard to prove that $\star_\s$ is independent
of $\s$, and we simply denote it $\star$.  The following theorem is
proved in section \ref{fgHc}.

\begin{theorem}\label{star}
  The star product $\star$ has the form
  \[f\star g = fg - \frac{i}{2} \{f,g\} + O(h^2).\]
\end{theorem}

We observe that this formula for the first-order term of $\star$
agrees with the first-order term of the star product constructed by
Andersen, Mattes and Reshetikhin in \cite{AMR2}, when we apply the
formula in Theorem \ref{star} to two holonomy functions
$h_{\gamma_1,\lambda_1}$ and $h_{\gamma_2,\lambda_2}$:
\[h_{\gamma_1,\lambda_1}\star h_{\gamma_2,\lambda_2} =
h_{\gamma_1\gamma_2,\lambda_1\cup\lambda_2} - \frac{i}{2}
h_{\{\gamma_1,\gamma_2\},\lambda_1\cup\lambda_2} + O(h^2).\] We recall
that $\{\gamma_1,\gamma_2\}$ is the Goldman bracket (see \cite{Go2})
of the two simple closed curves $\gamma_1$ and $\gamma_2$.

A similar result was obtained for the abelian case, i.e. in the case
where $M$ is the moduli space of flat $U(1)$-connections, by the first
author in \cite{A2}, where the agreement between the star product
defined in differential geometric terms and the star product of
Andersen, Mattes and Reshetikhin was proved to all orders.

We would finally also like to recall that the first named author has
shown that the Nielsen-Thurston classification of mapping classes is
determined by the Reshetikhin-Turaev TQFTs. We refer to \cite{A4.2}
for the full details of this.

A warm thanks is due to the editor of this volume for her persistent
encouragements towards the completion of this contribution.

\section{The Hitchin connection}\label{ghc}

In this section, we review our construction of the Hitchin connection
using the global differential geometric setting of \cite{A5}. This
approach is close in spirit to Axelrod, Della Pietra and Witten's in
\cite{ADW}, however we do not use any infinite dimensional gauge
theory. In fact, the setting is more general than the gauge theory
setting in which Hitchin in \cite{H} constructed his original
connection. But when applied to the gauge theory situation, we get the
corollary that Hitchin's connection agrees with Axelrod, Della Pietra
and Witten's.

Hence, we start in the general setting and let $(M,\omega)$ be any
compact symplectic manifold.

\begin{definition}\label{prequantumb}
  A prequantum line bundle $(\L, (\cdot,\cdot), \nabla)$ over the
  symplectic manifold $(M,\omega)$ consist of a complex line bundle
  $\L$ with a Hermitian structure $(\cdot,\cdot)$ and a compatible
  connection $\nabla$ whose curvature is
  \begin{align*}
    F_\nabla(X,Y) = [\nabla_X, \nabla_Y] - \nabla_{[X,Y]} = -i \omega
    (X,Y).
  \end{align*}
  We say that the symplectic manifold $(M,\omega)$ is prequantizable
  if there exist a prequantum line bundle over it.
\end{definition}

Recall that the condition for the existence of a prequantum line
bundle is that $[\frac{\omega}{2\pi}]\in \im(H^2(M,\bZ) \ra
H^2(M,\bR))$. Furthermore, the inequivalent choices of prequantum line
bundles (if they exist) are parametriced by $H^1(M,U(1))$ (see
e.g. \cite{Woodhouse}).

We shall assume that $(M,\omega)$ is prequantizable and fix a
prequantum line bundle $(\L, (\cdot,\cdot), \nabla)$.

Assume that $\T$ is a smooth manifold which smoothly parametrizes
K\"{a}hler structures on $(M,\omega)$. This means that we have a
smooth\footnote{Here a smooth map from $\T$ to $C^\infty(M,W)$, for
  any smooth vector bundle $W$ over $M$, means a smooth section of
  $\pi_M^*(W)$ over $\T\times M$, where $\pi_M$ is the projection onto
  $M$. Likewise, a smooth $p$-form on $\T$ with values in
  $C^\infty(M,W)$ is, by definition, a smooth section of
  $\pi_{\T}^*\Lambda^p(\T)\otimes \pi_M^*(W)$ over $\T\times M$. We
  will also encounter the situation where we have a bundle $\tW$ over
  $\T\times M$ and then we will talk about a smooth $p$-form on $\T$
  with values in $C^\infty(M,\tW_\s)$ and mean a smooth section of
  $\pi_{\T}^*\Lambda^p(\T)\otimes \tW$ over $\T\times M$.}  map $I :
\T \ra C^\infty(M,\End(TM))$ such that $(M,\omega, I_\s)$ is a
K\"{a}hler manifold for each $\s\in \T$.

We will use the notation $M_\sigma$ for the complex manifold $(M,
I_\s)$. For each $\s\in \T$, we use $I_\s$ to split the complexified
tangent bundle $TM_\bC$ into the holomorphic and the anti-holomorphic
parts. These we denote by
$$T_{\s} = E(I_\s,i) = \im(\Id - iI_\s)$$
and
$$\bT_{\s}= E(I_\s,-i) = \im(\Id + iI_\s)$$
respectively.

The real K\"{a}hler-metric $g_\s$ on $(M_\s,\omega)$, extended complex
linearly to $TM_\bC$, is by definition
\begin{align}
  \label{eq:3}
  g_\s(X,Y) = \omega(X,I_\s Y),
\end{align}
where $X,Y \in C^\infty(M,TM_\bC)$.

The divergence of a vector field $X$ is the unique function
$\delta(X)$ determined by
\begin{align}
  \label{eq:1}
  \mathcal{L}_X \omega^m = \delta(X) \omega^m.
\end{align}
It can be calculated by the formula $\delta(X) = \Lambda d (i_X
\omega)$, where $\Lambda$ denotes contraction with the K\"ahler form.
Eventhough the divergence only depends on the volume, which is
independent of the of the particular K\"ahler structure, it can be
expressed in terms of the Levi-Civita connection on $M_\sigma$ by
$\delta(X) = \tr \nabla_\sigma X$.

Inspired by this expression, we define the divergence of a symmetric
bivector field $B \in C^\infty(M, S^2(TM_{\bC}))$ by
\begin{align*}
  \delta_\sigma(B) = \tr \nabla_\sigma B.
\end{align*}
Notice that the divergence on bivector fields does depend on the point
$\sigma \in \mathcal{T}$.

Suppose $V$ is a vector field on $\T$. Then, we can differentiate $I$
along $V$ and we denote this derivative by $V[I] : \T \ra
C^\infty(M,\End(TM_\bC))$. Differentiating the equation $I^2 = -\Id$,
we see that $V[I]$ anti-commutes with $I$. Hence, we get that
\[V[I]_\s \in C^\infty(M, (T_\s^*\otimes \bT_\s)\oplus
(\bT_\s^*\otimes T_\s))\] for each $\s\in \T$. Let
\begin{align*}
  V[I]_\s = V[I]'_\s + V[I]''_\s
\end{align*}
be the corresponding decomposition such that $V[I]'_\s\in C^\infty(M,
\bT_\s^*\otimes T_\s)$ and $V[I]''_\s\in C^\infty(M, T_\s^*\otimes
\bT_\s)$.

Now we will further assume that $\T$ is a complex manifold and that
$I$ is a holomorphic map from $\T$ to the space of all complex
structures on $M$.  Concretely, this means that
\[V'[I]_\s = V[I]'_\s\] and
\[V''[I]_\s = V[I]''_\s\] for all $\s\in \T$, where $V'$ means the
$(1,0)$-part of $V$ and $V''$ means the $(0,1)$-part of $V$ over $\T$.

Let us define $\tG(V) \in C^\infty(M , TM_\bC\otimes TM_\bC)$ by
\[V[I] = \tG(V) \omega,\] and define $G(V) \in C^\infty(M, T_\s
\otimes T_\s)$ such that
\[\tG(V) = G(V) + {\overline G(V)} \]
for all real vector fields $V$ on $\T$. We see that $\tG$ and $G$ are
one-forms on $\T$ with values in $C^\infty(M , TM_\bC\otimes TM_\bC)$
and $C^\infty(M, T_\s \otimes T_\s)$, respectively.  We observe that
\[V'[I] = G(V)\omega,\] and $G(V) = G(V')$.

Using the relation \eqref{eq:3}, one checks that
\begin{align*}
  \tilde G(V) = - V[g^{-1}],
\end{align*}
where $g^{-1} \in C^\infty(M, S^2(TM))$ is the symmetric bivector
field obtained by raising both indices on the metric tensor.  Clearly,
this implies that $\tG$ takes values in $C^\infty(M , S^2(TM_\bC))$
and therefore that $G$ takes values in $C^\infty(M, S^2(T_\s))$.

On $\L^k$, we have the smooth family of $\bar\partial$-operators
$\nabla^{0,1}$ defined at $\s\in \T$ by
\[\nabla^{0,1}_\s = \frac12 (1+i I_\s)\nabla.\]
For every $\sigma\in \T$, we consider the finite-dimensional subspace
of $C^\infty(M,\L^k)$ given by
\[H_\sigma^{(k)} = H^0(M_\s, \L^k) = \{s\in C^\infty(M, \L^k)|
\nabla^{0,1}_\s s =0 \}.\] Let $\Nablat$ denote the trivial connection
in the trivial bundle $\mathcal{H}^{(k)} = \T\times C^\infty(M,\L^k)$,
and let $\D(M,\L^k)$ denote the vector space of differential operators
on $C^\infty(M,\L^k)$. For any smooth one-form $u$ on $\T$ with values
in $\D(M,\L^k)$, we have a connection $\Nabla$ in $\cH^{(k)}$ given by
$$\Nabla_V = \Nablat_V - u(V)$$
for any vector field $V$ on $\T$.

\begin{lemma}
  The connection $\Nabla$ in $\cH^{(k)}$ preserves the subspaces
  $H^{(k)}_\sigma \subset C^\infty(M,\L^k)$, for all $\sigma \in \T$,
  if and only if
  \begin{equation}
    \frac{i}2  V[I] \nabla^{1,0} s + \nabla^{0,1}u(V)s = 0\label{eqcond}
  \end{equation}
  for all vector fields $V$ on $\T$ and all smooth sections $s$ of
  $H^{(k)}$.
\end{lemma}

This result is not surprising. See \cite{A5} for a proof this
lemma. Observe that if this condition holds, we can conclude that the
collection of subspaces $H^{(k)}_\sigma \subset C^\infty(M,\L^k)$, for
all $\sigma \in \T$, form a subbundle $H^{(k)}$ of $\cH^{(k)}$.

We observe that $u(V'') = 0$ solves \eqref{eqcond} along the
anti-holomorphic directions on $\T$ since
\[V''[I] \nabla^{1,0} s = 0.\] In other words, the $(0,1)$-part of the
trivial connection $\Nablat$ induces a $\bar\partial$-operator on
$H^{(k)}$ and hence makes it a holomorphic vector bundle over $\T$.

This is of course not in general the situation in the $(1,0)$
direction. Let us now consider a particular $u$ and prove that it
solves \eqref{eqcond} under certain conditions.

On the K\"{a}hler manifold $(M_\s,\omega)$, we have the K\"{a}hler
metric and we have the Levi-Civita connection $\nabla$ in $T_\s$. We
also have the Ricci potential $F_\s\in C^\infty_0(M,\bR)$. Here
\[C^\infty_0(M,\bR) = \left\{ f\in C^\infty(M,\bR) \mid \int_M f
  \omega^m = 0\right\}, \] and the Ricci potential is the element of
$F_\s\in C^\infty_0(M,\bR)$ which satisfies
\[\Ric_\s = \Ric_\s^H + 2 i \partial_\s\dbar_\s F_\s,\]
where $\Ric_\s\in \Omega^{1,1}(M_\s)$ is the Ricci form and
$\Ric_\s^H$ is its harmonic part. We see that we get in this way a
smooth function $F : \T \ra C^\infty_0(M,\bR)$.

For any symmetric bivector field $B\in C^\infty(M, S^2(TM))$ we get a
linear bundle map
\begin{align*}
  B \colon TM^* \ra TM
\end{align*}
given by contraction. In particular, for a smooth function $f$ on $M$,
we get a vector field $B d f \in C^\infty(M,TM)$.

We define the operator
\begin{eqnarray*}
  \Delta_B &: &C^\infty(M,\L^k) \stackrel{\nabla}{\ra} C^\infty(M,TM^*\otimes\L^k)
  \stackrel{B\otimes\Id}{\ra}
  C^\infty(M,TM \otimes \L^k) \\
  && \qquad \stackrel{\nabla_\s\otimes \Id +
    \Id\otimes \nabla}{\ra}
  C^\infty(M,TM^* \otimes TM \otimes\L^k)
  \stackrel{\tr}{\ra}C^\infty(M,\L^k).
\end{eqnarray*}
Let's give a more concise formula for this operator.  Define the
operator
\begin{align*}
  \nabla^2_{X,Y} = \nabla_X \nabla_Y - \nabla_{\nabla_X Y},
\end{align*}
which is tensorial and symmetric in the vector fields $X$ and
$Y$. Thus, it can be evaluated on a symmetric bivector field and we
have
\begin{align*}
  \Delta_B = \nabla^2_B + \nabla_{\delta(B)}.
\end{align*}

Putting these constructions together, we consider, for some $n\in \bZ$
such that $2k+n \neq 0$, the following operator
\begin{equation}
  u(V) = \frac1{k+n/2}o(V) - V'[F],\label{equ}
\end{equation}
where
\begin{equation}
  o(V) = - \frac{1}{4} (\Delta_{G(V)} + 2\nabla_{G(V)dF} - 2n V'[F]).\label{eqo}
\end{equation}

The connection associated to this $u$ is denoted $\Nabla$, and we call
it the {\em Hitchin connection} in $\cH^{(k)}$.

\begin{definition}\label{Ridig}
  We say that the complex family $I$ of K\"{a}hler structures on
  $(M,\omega)$ is {\em Rigid} if
  \[\dbar_\sigma (G(V)_\sigma) = 0 \]
  for all vector fields $V$ on $\T$ and all points $\sigma\in \T$.
\end{definition}

We will assume our holomorphic family $I$ is rigid.

\begin{theorem}[Andersen]\label{HCE}
  Suppose that $I$ is a rigid family of K\"{a}hler structures on the
  compact, prequantizable symplectic manifold $(M,\omega)$ which
  satisfies that there exists an $n\in \bZ$ such that the first Chern
  class of $(M,\omega)$ is $n [\frac{\omega}{2\pi}]\in H^2(M,\bZ)$ and
  $H^1(M,\bR) = 0$. Then $u$ given by \eqref{equ} and \eqref{eqo}
  satisfies \eqref{eqcond}, for all $k$ such that $2k+n \neq 0$.
\end{theorem}

Hence, the Hitchin connection $\Nabla$ preserves the subbundle
$H^{(k)}$ under the stated conditions. Theorem \ref{HCE} is
established in \cite{A5} through the following three lemmas.

\begin{lemma} \label{dbarl} Assume that the first Chern class of
  $(M,\omega)$ is $n [\frac{\omega}{2\pi}]\in H^2(M,\bZ)$. For any
  $\s\in \T$ and for any $G\in H^0(M_\s, S^2(T_\s))$, we have the
  following formula
  \begin{align*}
    \nabla^{0,1}_\s (\Delta_G(s) + 2 \nabla_{G d F_\s}(s)) = - i (2 k
    + n) \omega G \nabla (s) + 2 ik (G dF_\sigma) \omega + ik
    \delta_\sigma(G)\omega) s,
  \end{align*}
  for all $s\in H^0(M_\s, \L^k)$.
\end{lemma}

\begin{lemma}\label{Vriccipot}
  We have the following relation
  \begin{align*}
    4i \bar\partial_\s (V'[F]_\s) = 2 (G(V) dF)_\sigma \omega +
    \delta_\sigma (G(V))_\sigma \omega,
  \end{align*}
  provided that $H^1(M,\bR) = 0$.
\end{lemma}

\begin{lemma}\label{lemma4}
  For any smooth vector field $V$ on $\T$, we have that
  \begin{equation}
    2(V'[\Ric])^{1,1} =  \partial (\delta(G(V)) \omega).
  \end{equation}
\end{lemma}

Let us here recall how Lemma \ref{Vriccipot} is derived from Lemma
\ref{lemma4}. By the definition of the Ricci potential
\[\Ric = \Ric^H + 2 i d \bar \partial F,\]
where $\Ric^H = n\omega$ by the assumption $c_1(M, \omega) =
n[\frac{\omega}{2\pi}]$. Hence
\[V'[\Ric] = - d V'[I] d F + 2 i d \bar \partial V'[F],\] and
therefore
\[ 4i \partial \bar \partial V'[F] = 2(V'[\Ric])^{1,1} + 2\partial
V'[I] d F.\] From the above, we conclude that
$$ (2 (G(V) d F) \omega + \delta(G(V)) \omega -
4i \bar \partial V'[F])_\sigma \in \Omega^{0,1}_\s(M)$$ is a
$\partial_\sigma$-closed one-form on $M$. From Lemma \ref{dbarl}, it
follows that it is also $\bar \partial_\sigma$-closed, whence it must
be a closed one-form. Since we assume that $H^1(M,\bR) = 0$, we see
that it must be exact. But then it in fact vanishes since it is of
type $(0,1)$ on $M_\s$.

From the above we conclude that
$$
u(V) = \frac1{k+n/2}o(V) - V'[F] = - \frac1{4k+2n} \left\{
  \Delta_{G(V)} + 2\nabla_{G(V)dF} + 4k V'[F]\right\}
$$
solves \eqref{eqcond}. Thus we have established Theorem \ref{HCE} and
hence Theorem \ref{MainGHCI}.

In \cite{AGL} we use half-forms and the metaplectic correction to
prove the existence of a Hitchin connection in the context of
half-form quantization. The assumption that the first Chern class of
$(M,\omega)$ is $n [\frac{\omega}{2\pi}]\in H^2(M,\bZ)$ is then just
replaced by the vanishing of the second Stiefel-Whitney class of $M$
(see \cite{AGL} for more details).

Suppose $\Gamma$ is a group which acts by bundle automorphisms of $\L$
over $M$ preserving both the Hermitian structure and the connection in
$\L$. Then there is an induced action of $\Gamma$ on $(M,\omega)$. We
will further assume that $\Gamma$ acts on $\T$ and that $I$ is
$\Gamma$-equivariant. In this case we immediately get the following
invariance.

\begin{lemma}
  The natural induced action of $\Gamma$ on $\cH^{(k)}$ preserves the
  subbundle $H^{(k)}$ and the Hitchin connection.
\end{lemma}

We are actually interested in the induced connection $\Nablae$ in the
endomorphism bundle $\End(H^{(k)})$. Suppose $\Phi$ is a section of
$\End(H^{(k)})$. Then for all sections $s$ of $H^{(k)}$ and all vector
fields $V$ on $\T$, we have that
\[(\Nablae_V \Phi) (s) = \Nabla_V \Phi(s) - \Phi(\Nabla_V(s)).\]
Assume now that we have extended $\Phi$ to a section of $\Hom
(\cH^{(k)},H^{(k)})$ over $\T$. Then
\begin{equation}\label{endocon}
  \Nablae_V \Phi = \Nablaet_V \Phi + [\Phi, u(V)],
\end{equation}
where $\Nablaet$ is the trivial connection in the trivial bundle
$\End(\cH^{(k)})$ over $\T$.

\section{Toeplitz operators and Berezin-Toeplitz deformation
  quantization}\label{BZdq}

We shall in this section discuss the Toeplitz operators and their
asymptotics as the level $k$ goes to infinity. The properties we need
can all be derived from the fundamental work of Boutet de Monvel and
Sj\"{o}strand. In \cite{BdMS}, they did a microlocal analysis of the
Szeg\"{o} projection which can be applied to the asymptotic analysis
in the situation at hand, as it was done by Boutet de Monvel and
Guillemin in \cite{BdMG} (in fact in a much more general situation
than the one we consider here) and others following them. In
particular, the applications developed by Schlichenmaier and further
by Karabegov and Schlichenmaier to the study of Toeplitz operators in
the geometric quantization setting is what will interest us here. Let
us first describe the basic setting.

For each $f\in C^\infty(M)$, we consider the prequantum operator,
namely the differential operator $M_f^{(k)}: C^\infty(M,L^k) \ra
C^\infty(M,L^k)$ given by
\[M_f^{(k)}(s) = f s\] for all $s\in H^0(M,L^k)$.

These operators act on $C^\infty(M,\L^k)$ and therefore also on the
bundle $\cH^{(k)}$, however, they do not preserve the subbundle
$H^{(k)}$. In order to turn these operators into operators which acts
on $H^{(k)}$ we need to consider the Hilbert space structure.

Integrating the inner product of two sections against the volume form
associated to the symplectic form gives the pre-Hilbert space
structure on $C^\infty(M,\L^k)$
\[\langle s_1,s_2\rangle = \frac1{m!}\int_M (s_1,s_2) \omega^m.\]
We think of this as a pre-Hilbert space structure on the trivial
bundle $\cH^{(k)}$ which of course is compatible with the trivial
connection in this bundle.  This pre-Hilbert space structure induces a
Hermitian structure $\langle\cdot,\cdot\rangle$ on the finite rank
subbundle $H^{(k)}$ of $\cH^{(k)}$.  The Hermitian structure
$\langle\cdot,\cdot\rangle$ on $H^{(k)}$ also induces the operator
norm $\|\cdot\|$ on $\End(H^{(k)})$.

Since $H^{(k)}_\s$ is a finite dimensional subspace of
$C^\infty(M,\L^k)= \cH_\s^{(k)}$ and therefore closed, we have the
orthogonal projection $\pi_\s^{(k)} \colon \cH_\s^{(k)} \ra
H^{(k)}_\s$. Since $H^{(k)}$ is a smooth subbundle of $\cH^{(k)}$, the
projections $\pi_\s^{(k)}$ form a smooth map $\pi^{(k)}$ from $\T$ to
the space of bounded operators on the $L_2$-completion of
$C^\infty(M,\L^k)$. The easiest way to see this is to consider a local
frame $(s_1, \ldots s_{\rank H^{(k)}})$ of $H^{(k)}$. Let $h_{ij} =
\langle s_i, s_j\rangle$, and let $h^{-1}_{ij}$ be the inverse matrix
of $h_{ij}$. Then
\begin{equation}\label{projf}
  \pi_\s^{(k)}(s) = \sum_{i,j}\langle s, (s_i)_\s\rangle (h^{-1}_{ij})_\s (s_j)_\s.
\end{equation}

From these projections, we can construct the Toeplitz operators
associated to any smooth function $f\in C^\infty(M)$. It is the
operator $T^{(k)}_{f,\s} : \cH_\s^{(k)} \ra H^{(k)}_\s$ defined by
\[T_{f,\s}^{(k)}(s) = \pi_\s^{(k)}(fs)\] for any element $s$ in
$\cH_\s^{(k)}$ and any point $\s\in \T$.  We observe that the Toeplitz
operators are smooth sections $T_{f}^{(k)}$ of the bundle
$\Hom(\cH^{(k)},H^{(k)})$ and restrict to smooth sections of
$\End(H^{(k)})$.

\begin{remark}
  Similarly, for any Pseudo-differential operator $A$ on $M$ with
  coefficients in $\L^k$ (which may even depend on $\s\in \T$), we can
  consider the associated Toeplitz operator $\pi^{(k)} A$ and think of
  it as a section of $\Hom(\cH^{(k)},H^{(k)})$. However, whenever we
  consider asymptotic expansions of such or consider their operator
  norms, we implicitly restrict them to $H^{(k)}$ and consider them as
  sections of $\End(H^{(k)})$ or equivalently assume that they have
  been precomposed with $\pi^{(k)}$.

\end{remark}

Suppose that we have a smooth section $X\in C^\infty(M, T_\s)$ of the
holomorphic tangent bundle of $M_\s$. We then claim that the operator
$\pi^{(k)} \nabla_X$ is a zero-order Toeplitz operator. Supposing that
$s_1\in C^\infty(M,\L^k)$ and $s_2 \in H^0(M_\s,\L^k)$, we have that
\[X(s_1,s_2) = (\nabla_X s_1, s_2).\] Now, calculating the Lie
derivative along $X$ of $(s_1,s_2)\omega^m$ and using the above, one
obtains after integration that
\begin{align*}
  \langle \nabla_X s_1, s_2 \rangle = - \langle \delta(X) s_1, s_2
  \rangle,
\end{align*}
Thus
\begin{equation}\pi^{(k)} \nabla_X = -
  T_{\delta(X)}^{(k)},\label{1to0order}
\end{equation}
as operators from $C^\infty(M,L^k)$ to $H^0(M,L^k)$.

Iterating \eqref{1to0order}, we find for all $X_1,X_2 \in
C^\infty(M,T_\s)$ that
\begin{equation}\pi^{(k)} \nabla_{X_1}\nabla_{X_2} =
  T^{(k)}_{\delta(X_2)\delta(X_1)
    + X_2[\delta(X_1)]},\label{2to0order}
\end{equation}
again as operators from $C^\infty(M,\L^k)$ to $H^0(M_\s,\L^k)$.

We calculate the adjoint of $\nabla_X$ for any complex vector field
$X\in C^\infty(M, TM_{\bC})$. For $s_1, s_2 \in C^\infty(M,\L^k)$, we
have that
\[\bar X (s_1,s_2) = (\nabla_{\bar X} s_1, s_2) + ( s_1, \nabla_X
s_2).\] Computing the Lie derivative along $\bar X$ of
$(s_1,s_2)\omega^m$ and integrating, we get that
\[\langle \nabla_{\bar X} s_1, s_2 \rangle + \langle (\nabla_{X})^*
s_1, s_2 \rangle = - \langle \delta(\bar X) s_1, s_2 \rangle.\] Hence,
we see that
\begin{align}
  \label{eq:7}
  (\nabla_{X})^* = - \nabla_{\bar X} - \delta(\bar X)
\end{align}
as operators on $C^\infty(M, \L^k)$. In particular, if $X \in
C^\infty(M, T_\sigma)$ is a section of the holomorphic tangent bundle,
we see that
\begin{equation}\label{1to0order*}
  \pi^{(k)} (\nabla_{X})^* \pi^{(k)} = - T^{(k)}_{\delta(\bar X)} \vert _{H^0(M_\s,\L^k)},
\end{equation}
again as operators on $H^0(M_\s,\L^k)$.

The product of two Toeplitz operators associated to two smooth
functions will in general not be the Toeplitz operator associated to a
smooth function again.  But, by the results of Schlichenmaier
\cite{Sch}, there is an asymptotic expansion of the product in terms
of such Toeplitz operators on a compact K{\"a}hler manifold.

\begin{theorem}[Schlichenmaier]\label{S}
  For any pair of smooth functions $f_1, f_2\in \C^\infty(M)$, we have
  an asymptotic expansion
  \[T_{f_1,\s}^{(k)}T_{f_2,\s}^{(k)} \sim \sum_{l=0}^\infty
  T_{c_\s^{(l)}(f_1,f_2),\s}^{(k)} k^{-l},\] where
  $c_\s^{(l)}(f_1,f_2) \in C^\infty(M)$ are uniquely determined since
  $\sim$ means the following: For all $L\in \Z_+$ we have that
  \[\|T_{f_1,\s}^{(k)}T_{f_2,\s}^{(k)} - \sum_{l=0}^L
  T_{c_\s^{(l)}(f_1,f_2),\s}^{(k)} k^{-l}\| = O(k^{-(L+1)})\]
  uniformly over compact subsets of $\T$.  Moreover,
  $c_\s^{(0)}(f_1,f_2) = f_1f_2$.
\end{theorem}

\begin{remark} \label{rem:1} It will be useful for us to define new
  coefficients ${\tilde c}_\s^{(l)}(f,g) \in C^\infty(M)$ which
  correspond to the expansion of the product in $1/(k+n/2)$ (where $n$
  is some fixed integer):
  \[T_{f_1,\s}^{(k)}T_{f_2,\s}^{(k)} \sim \sum_{l=0}^\infty T_{{\tilde
      c}_\s^{(l)}(f_1,f_2),\s}^{(k)} (k+n/2)^{-l}.\] For future
  reference, we note that the first three coefficients are given by
  $\tilde c^{(0)}_\sigma(f_1,f_2) = c^{(0)}_\sigma(f_1,f_2)$, $\tilde
  c^{(1)}_\sigma(f_1,f_2) = c^{(1)}_\sigma(f_1,f_2)$ and $\tilde
  c^{(2)}_\sigma(f_1,f_2) = c^{(2)}_\sigma(f_1,f_2) + \frac{n}{2}
  c^{(1)}_\sigma(f_1,f_2)$.
\end{remark}

Theorem \ref{S} is proved in \cite{Sch} where it is also proved that
the formal generating series for the $c_\s^{(l)}(f_1,f_2)$'s gives a
formal deformation quantization of symplectic manifold $(M,\omega)$.

We recall the definition of a formal deformation
quantization. Introduce the space of formal functions $C^\infty_h(M) =
C^\infty(M)[[h]]$ as the space of formal power series in the variable
$h$ with coefficients in $C^\infty(M)$. Let $\bC_h = \bC[[h]]$ denote
the formal constants.

\begin{definition}
  A deformation quantization of $(M,\omega)$ is an associative product
  $\star$ on $C^\infty_h(M)$ which respects the $\bC_h$-module
  structure.  For $f,g \in C^\infty(M)$, it is defined as
  \[f \star g = \sum_{l=0}^\infty c^{(l)}(f,g) h^{l},\] through a
  sequence of bilinear operators
$$c^{(l)} : C^\infty(M) \otimes C^\infty(M) \ra C^\infty(M),$$
which musth satisfy
\begin{align*}
  c^{(0)}(f, g) = fg \qquad \text{and} \qquad c^{(1)}(f, g) -
  c^{(1)}(g,f) = -i \{f, g\}.
\end{align*}
The deformation quantization is said to be differential if the
operators $c^{(l)}$ are bidifferential operators. Considering the
symplectic action of $\Gamma$ on $(M,\omega)$, we say that a star
product is $\Gamma$-invariant if
$$ \gamma^*(f \star g) = \gamma^*(f) \star \gamma^*(g)$$
for all $f,g\in C^\infty(M)$ and all $\gamma\in\Gamma$.
\end{definition}

\begin{theorem}[Karabegov \& Schlichenmaier]\label{tKS1}
  The product $\BTstar_\s$ given by
  \[f \BTstar_\s g = \sum_{l=0}^\infty c_\s^{(l)}(f,g) h^{l},\] where
  $f,g \in C^\infty(M)$ and $c_\s^{(l)}(f,g)$ are determined by
  Theorem \ref{S}, is a differentiable deformation quantization of
  $(M,\omega)$.
\end{theorem}

\begin{definition}
  The Berezin-Toeplitz deformation quantization of the compact
  K{\"a}hler manifold $(M_\s,\omega)$ is the product $\BTstar_\s$.
\end{definition}

\begin{remark} Let $\Gamma_\s$ be the $\s$-stabilizer subgroup of
  $\Gamma$.  For any element $\gamma\in \Gamma_\s$, we have that
  \[\gamma^*(T^{(k)}_{f,\s}) = T^{(k)}_{\gamma^*f,\s}.\]
  This implies the invariance of $\BTstar_\s$ under the
  $\s$-stabilizer $\Gamma_\s$.
\end{remark}

\begin{remark} Using the coefficients from Remark \ref{rem:1}, we
  define a new star product by
  \[f\tBTstar_\s g = \sum_{l=0}^\infty {\tilde c}_\s^{(l)}(f,g)
  h^{l}.\] Then
  \[f\tBTstar_\s g = \left( (f\circ \phi^{-1})\BTstar_\s (g\circ
    \phi^{-1})\right) \circ \phi\] for all $f,g\in C_h^\infty(M)$,
  where $\phi(h) = \frac{2h}{2 + nh}$.
\end{remark}

\section{The formal Hitchin connection}\label{fgHc}

In this section, we study the the formal Hitchin. We assume the
conditions on $(M,\omega)$ and $I$ of Theorem \ref{HCE}, thus
providing us with a Hitchin connection $\Nabla$ in $H^{(k)}$ over $\T$
and the associated connection $\Nablae$ in $\End(H^{(k)})$.

Recall from the introduction the definition of a formal connection in
the trivial bundle of formal functions. Theorem \ref{MainFGHCI2},
establishes the existence of a unique formal Hitchin connection,
expressing asymptotically the interplay between the Hitchin connection
and the Toeplitz operators.

We want to give an explicit formula for the formal Hitchin connection
in terms of the star product $\tilde \star^{BT}$. We recall that in
the proof of Theorem \ref{MainFGHCI2}, given in \cite{A5}, it is shown
that the formal Hitchin connection is given by
\begin{equation}
  \tD(V)(f) =  - V[F]f + V[F]\tBTstar f + h (E(V)(f) -
  H(V) \tBTstar f), \label{formalcon}
\end{equation}
where $E$ is the one-form on $\T$ with values in $\D(M)$ such that
\begin{align}
  \label{eq:2}
  T^{(k)}_{E(V)f} = \pi^{(k)} o(V)^* f \pi^{(k)} + \pi^{(k)} f o(V)
  \pi^{(k)},
\end{align}
and $H$ is the one form on $\T$ with values in $C^\infty(M)$ such that
$H(V) = E(V)(1)$. Thus, we must find an explicit expression for the
operator $E(V)$.

The following lemmas will prove helpful.
\begin{lemma}
  \label{lem:2}
  The adjoint of $\Delta_B$ is given by
  \begin{align*}
    \Delta^*_B = \Delta_{\bar B},
  \end{align*}
  for any (complex) symmetric bivector field $B \in C^\infty(M,
  S^2(TM_\bC))$.
\end{lemma}

\begin{proof}
  First, we write $B = \sum_r^R X_r \otimes Y_r$. Then
  \begin{align*}
    \Delta_B = \sum_r^R \nabla_{X_r} \nabla_{Y_r} +
    \nabla_{\delta(X_r) Y_r}.
  \end{align*}
  Now, using \eqref{eq:7}, we get that
  \begin{align*}
    (\nabla_{X_r} \nabla_{Y_r})^* &= (\nabla_{Y_r})^* (\nabla_{X_r})^*
    = (\nabla_{\bar Y_r} + \delta(\bar Y_r) (\nabla_{\bar X_r} +
    \delta(\bar X_r)) \\ &= \nabla_{\bar Y_r} \nabla_{\bar X_r} +
    \nabla_{\bar Y_r} \delta(X_r) + \delta(\bar Y_r) \nabla_{\bar X_r}
    + \delta(\bar Y_r)\delta(\bar X_r),
  \end{align*}
  and
  \begin{align*}
    (\nabla_{\delta(X_r) Y_r})^* &= - \nabla_{\delta(\bar X) \bar Y_r}
    - \delta(\delta(\bar X_r) \bar Y_r) \\ &= - \delta(\bar X_r)
    \nabla_{\bar Y_r} - \bar Y_r[\delta(\bar X_r)] - \delta(\bar
    X_r)\delta(\bar Y_r) \\ &= - \nabla_{\bar Y_r} \delta(\bar X_r) -
    \delta(\bar X_r)\delta(\bar Y_r),
  \end{align*}
  so we conclude that
  \begin{align*}
    \Delta_B^* = \sum_r^R \nabla_{\bar Y_r} \nabla_{\bar X_r} +
    \delta(\bar Y_r) \nabla_{\bar X_r} = \Delta_{\bar B},
  \end{align*}
  since $B$ is symmetric.
\end{proof}

\begin{lemma}
  \label{lem:1}
  The operator $\Delta_B$ satisfies
  \begin{align*}
    \pi^{(k)} \Delta_B s = 0,
  \end{align*}
  for any section $s \in C^\infty(M, \mathcal{L}^k)$ and any symmetric
  bivector field $B$.
\end{lemma}
\begin{proof}
  Again, we write $B = \sum^R_r X_r \otimes Y_r$ and recall from
  \eqref{2to0order} that
  \begin{align*}
    \pi^{(k)} \nabla_{X_r} \nabla_{Y_r} s = \pi^{(k)} (
    \delta(X_r)\delta(Y_r) + Y_r[\delta(X_r)]) s.
  \end{align*}
  On the other hand, we have that
  \begin{align*}
    \pi^{(k)} \nabla_{\delta(X_r) Y_r} s = - \pi^{(k)}\delta
    (\delta(X_r) Y_r) s = - \pi^{(k)}(\delta(X_r)\delta(Y_r) +
    Y_r[\delta(X_r)]) s,
  \end{align*}
  and it follows immediately that
  \begin{align*}
    \pi^{(k)} \Delta_B s = \pi^{(k)} \sum_r^R (\nabla_{X_r}
    \nabla_{Y_r} + \nabla_{\delta(X_r) Y_r}) s = 0,
  \end{align*}
  which proves the lemma.
\end{proof}

Finally, it will prove useful to observe that
\begin{align}
  \label{eq:4}
  \delta( B df) = \Delta_B (f),
\end{align}
for any function $f$ and any bivector field $B$.

Now, the adjoint of $o(V)$ is given by
\begin{align*}
  o(V)^* &= - \frac{1}{4}( \Delta_{\bar G(V)} - 2\nabla_{\bar G(V) dF}
  - 2 \Delta_{\bar G(V)}(F) - 2n V''[F]),
\end{align*}
where we used \eqref{eq:4}. Furthermore, we observe that $o(V)^*$
differentiates in anti-holomorphic directions only, which implies that
\begin{align*}
  \pi^{(k)} o(V)^* f \pi^{(k)} &= \pi^{(k)} o(V)^*(f) \pi^{(k)} \\ &=
  - \frac{1}{4} \pi^{(k)} (\Delta_{\bar G(V)}(f) - 2\nabla_{\bar G(V)
    dF}(f) - 2\Delta_{\bar G(V)}(F)f - 2n V''[F]f) \pi^{(k)}.
\end{align*}
This gives an explicit formula for the first term of \eqref{eq:2}.

To determine the second term of \eqref{eq:2}, we observe that
\begin{align*}
  \Delta_{G(V)} f s = f \Delta_{G(V)} s + \Delta_{G(V)}(f) s + 2
  \nabla_{G(V)df} s.
\end{align*}
Projecting both sides onto the holomorphic sections and applying Lemma
\ref{lem:1} and the formula \eqref{eq:4}, we get that
\begin{align*}
  \pi^{(k)} f \Delta_{G(V)} &= - \pi^{(k)} (\Delta_{G(V)}(f) +
  2\nabla_{G(V)df}) = \pi^{(k)} \Delta_{G(V)}(f).
\end{align*}
Furtermore, observe that
\begin{align*}
  \pi^{(k)} f \nabla_{G(V) dF} &= \pi^{(k)} (\nabla_{G(V) dF} f -
  \nabla_{G(V) dF}(f) ) \\ &= - \pi^{(k)} (\nabla_{G(V)dF}(f) +
  \Delta_{G(V)}(F)f),
\end{align*}
where we once again used \eqref{eq:4} for the last equality. Thus, we
get that
\begin{align*}
  \pi^{(k)} f o(V) \pi^{(k)} &= -\frac{1}{4} \pi^{(k)} ( \Delta_{G(V)}
  (f) - 2\nabla_{G(V)dF}(f) - 2\Delta_{G(V)}(F) f - 2nV'[F] f )
  \pi^{(k)},
\end{align*}
which gives an explicit formula for the second term of \eqref{eq:2}.
Finally, we can conclude that
\begin{align*}
  E(V)(f) = - \frac{1}{4} (\Delta_{\tilde G(V)}(f) - 2\nabla_{\tilde
    G(V)dF}(f) - 2\Delta_{\tilde G(V)}(F) f - 2n V[F] f),
\end{align*}
satisfies \eqref{eq:2} and hence \eqref{formalcon}. Also, we note that
\begin{align*}
  H(V) = E(V)(1) = \frac{1}{2} (\Delta_{\tilde G(V)}(F) + n V[F]).
\end{align*}
Summarizing the above, we have proved the following
\begin{theorem}
  \label{thm:1}
  The formal Hitchin connection is given by
  \begin{align*}
    D_V f &= V[f] - \frac{1}{4}h \Delta_{\tilde G(V)}(f) +
    \frac{1}{2}h \nabla_{\tilde G(V)dF}(f) + V[F] \tBTstar f - V[F] f
    \\ & \quad - \frac{1}{2} h(\Delta_{\tilde G(V)}(F) \tBTstar f +
    nV[F] \tBTstar f - \Delta_{\tilde G(V)}(F) f - n V[F]f)
  \end{align*}
  for any vector field $V$ and any section $f$ of $C_h$.
\end{theorem}

The next lemma is also proved in \cite{A5}, and it follows basically
from the fact that
\[\Nablae_V (T^{(k)}_{f}T^{(k)}_{g}) =
\Nablae_V (T^{(k)}_{f})T^{(k)}_{g} + T^{(k)}_{f}\Nablae_V
(T^{(k)}_{g}). \] We have

\begin{lemma}\label{Deriv}
  The formal operator $D_V$ is a derivation for $\tBTstar_\sigma$ for
  each $\s\in\T$, i.e.
  \[D_V(f\tBTstar g) = D_V(f)\tBTstar g + f\tBTstar D_V(g)\] for all
  $f,g\in C^\infty(M)$.
\end{lemma}

If the Hitchin connection is projectively flat, then the induced
connection in the endomorphism bundle is flat and hence so is the
formal Hitchin connection by Proposition 3 of \cite{A5}.

Recall from Definition \ref{formaltrivi2} in the introduction the
definition of a formal trivialization. As mentioned there, such a
formal trivialization will not exist even locally on $\T$, if $D$ is
not flat. However, if $D$ is flat, then we have the following result.

\begin{proposition}\label{Dfc}
  Assume that $D$ is flat and that $\tD = 0$ mod $h$. Then locally
  around any point in $\T$, there exists a formal trivialization. If
  $H^1(\T,\bR) = 0$, then there exists a formal trivialization defined
  globally on $\T$. If further $H^1_\Gamma(\T,D(M)) = 0$, then we can
  construct $P$ such that it is $\Gamma$-equivariant.
\end{proposition}

In this proposition, $H^1_\Gamma(\T,D(M))$ simply refers to the
$\Gamma$-equivariant first \mbox{de Rham} cohomology of $\T$ with
coefficients in the real $\Gamma$-vector space $D(M)$.

Now suppose we have a formal trivialization $P$ of the formal Hitchin
connection $D$.  We can then define a new smooth family of star
products, parametrized by $\T$, by
\[f\star_\s g = P_\s^{-1}(P_\s(f) \tBTstar_\s P_\s(g))\] for all
$f,g\in C^\infty(M)$ and all $\s\in \T$. Using the fact that $P$ is a
trivialization, it is not hard to prove.

\begin{proposition}\label{ftrdq}
  The star products $\star_\s$ are independent of $\s\in\T$.
\end{proposition}

Then, we have the following which is proved in \cite{A5}.

\begin{theorem}[Andersen]\label{general}
  Assume that the formal Hitchin connection $D$ is flat and
$$H^1_\Gamma(\T,D(M)) = 0,$$ then there is a $\Gamma$-invariant
trivialization $P$ of $D$ and the star product
\[f\star g = P_\s^{-1}(P_\s(f) \tBTstar_\s P_\s(g))\] is independent
of $\s\in \T$ and $\Gamma$-invariant. If $H^1_\Gamma(\T,C^\infty(M)) =
0$ and the commutant of $\Gamma$ in $D(M)$ is trivial, then a
$\Gamma$-invariant differential star product on $M$ is unique.
\end{theorem}

We calculate the first term of the equivariant formal trivialization
of the formal Hitchin connection. Let $f$ be any function on $M$ and
suppose that $P(f) = \sum_l \tilde f_l h^l$ is parallel with respect
to the formal Hitchin connection. Thus, we have that
\begin{align*}
  0 = D_V P(f) = h(V[\tilde f_1] - \frac{1}{4} \Delta_{\tilde G(V)}
  (\tilde f_0) + \frac{1}{2}\nabla_{\tilde G(V)dF}(\tilde f_0) +
  c^{(1)}(V[F], \tilde f_0)) + O(h^2).
\end{align*}
But $\tilde f_0 = f$, and so we get in particular
\begin{align}
  \label{eq:5}
  0 = V[\tilde f_1] - \frac{1}{4} \Delta_{\tilde G(V)} (f) +
  \frac{1}{2}\nabla_{\tilde G(V)dF}(f) + c^{(1)}(V[F], f).
\end{align}
By the results of \cite{KS}, $\BTstar$ is a differential star product
with seperation of variables, in the sense that it only differentiates
in holomorphic directions in the first entry and antiholomorphic
directions in the second. As argued in \cite{Kar}, all such star
products have the same first order coefficient, namely
\begin{align}\
  \label{eq:8}
  c^{(1)}(f_1, f_2) = -g(\partial f_1, \bar \partial f_2) = i
  \nabla_{X_{f_1}''} (f_2)
\end{align}
for any functions $f_1, f_2 \in C^\infty(M)$. From this, it is easily
seen that
\begin{align*}
  V[c^{(1)}](f_1, f_2) = \frac{1}{2} df_1 \tilde G(V) df_2 =
  \frac{1}{2} \nabla_{\tilde G(V) df_1} f_2.
\end{align*}
Applying this to \eqref{eq:5}, we see that
\begin{align}
  \label{eq:6}
  V[\tilde f_1] = \frac{1}{4} \Delta_{\tilde G(V)} (f) - V[c^{(1)}(F,
  f)].
\end{align}
But the variation of the Laplace-Beltrami operator is given by
\begin{align*}
  V[\Delta] f = V[\delta(g^{-1}d f)] = \delta(V[g^{-1}] df) = -
  \delta(\tilde G(V)df) = - \Delta_{\tilde G(V)} f,
\end{align*}
and so we conclude that
\begin{align*}
  V[\tilde f_1] = - V [\frac{1}{4} \Delta f + c^{(1)}(F, f)].
\end{align*}
We have thus proved
\begin{proposition}
  \label{prop:1}
  When it exists, the equivariant formal trivialization of the formal
  Hitchin connection has form
  \begin{align*}
    P = \id - h(\frac{1}{4} \Delta + i\nabla_{X''_F}) + O(h^2).
  \end{align*}
\end{proposition}
Using this proposition, one easily calculates that
\begin{align*}
  P(f_1) \tBTstar P(f_2) &= f_1f_2 - h(\frac{1}{4} f_1 \Delta f_2 +
  \frac{1}{4} f_2 \Delta f_1 + i\nabla_{X''_F} f_1 + i\nabla_{X''_F}
  f_2) \\ & \quad + h c^{(1)}(f_1, f_2) + O(h^2).
\end{align*}
Finally, using the explicit formula \eqref{eq:8} for $c^{(1)}$, we get
that
\begin{align*}
  P^{-1}(P(f_1) \tBTstar P(f_2)) &= f_1f_2 - h g(\partial f_1,
  \bar \partial f_2) + \frac{1}{2}h g(df_1, df_2) + O(h^2) \\ &=
  f_1f_2 - h \frac{1}{2} (g(\partial f_1,
  \bar \partial f_2) - g(\bar \partial f_1, \partial f_2)) + O(h^2) \\
  &= f_1 f_2 - ih \frac{1}{2} \{f_1, f_2\} + O(h^2).
\end{align*}
This proves Theorem \ref{star}.

\nocite{*} \bibliographystyle{OUPnamed} \bibliography{biblist}

\begin{thebibliography}{0}


\bibitem[A1]{A1.5} J.E. Andersen, \emph{New polarizations on the
    moduli space and the Thurston compactification of Teichmuller
    space}, International Journal of Mathematics, {\bf 9}, No.1
  (1998), 1--45.



\bibitem[A2]{A2} J.E. Andersen, \emph{Geometric quantization and
    deformation quantization of abelian moduli spaces},
  Commun. Math. Phys. {\bf 255} (2005), 727--745.

\bibitem[A3]{A3} J. E. Andersen, \emph{Asymptotic faithfulness of the
    quantum $\SU(n)$ representations of the mapping class groups}.
  Annals of Mathematics, {\bf 163} (2006), 347--368.


\bibitem[A4]{A4} J.E. Andersen, \emph{Asymptotic faithfulness of the
    quantum $\SU(n)$ representations of the mapping class groups in
    the singular case}, In preparation.

\bibitem[A5]{A4.2} J.E. Andersen, \emph{The Nielsen-Thurston
    classification of mapping classes is determined by TQFT},
  math.QA/0605036.

\bibitem[A6]{A5} J.E. Andersen, \emph{Hitchin's connection, Toeplitz
    operators and symmetry invariant deformation quantization},
  math.DG/0611126.



\bibitem[A7]{A6} J.E. Andersen, \emph{Asymptotic in Teichmuller space
    of the Hitchin connection}, In preparation.

\bibitem[A8]{A7} J.E. Andersen, \emph{Mapping Class Groups do not have
    Kazhdan's Property (T)}, math.QA/0706.2184.

\bibitem[AC]{AC}J.E. Andersen \& M. Christ, \emph{Asymptotic expansion
    of the Szeg\"{o} kernel on singular algebraic varieties}, In
  preparation.

\bibitem[AGL]{AGL} J.E. Andersen, M. Lauritsen \& N. L. Gammelgaard,
  \emph{Hitchin's Connection in Half-form Quantization},
  arXiv:0711.3995.


\bibitem[AMU]{AMU} J.E. Andersen, G. Masbaum \& K. Ueno,
  \emph{Topological quantum field theory and the Nielsen-Thurston
    classification of $M(0,4)$}, Math. Proc. Cambridge
  Philos. Soc. 141 (2006), no. 3, 477--488.



\bibitem[AMR1]{AMR1} J.E. Andersen, J. Mattes \& N. Reshetikhin,
  \emph{The Poisson Structure on the Moduli Space of Flat Connections
    and Chord Diagrams}. Topology {\bf 35}, pp.1069--1083 (1996).


\bibitem[AMR2]{AMR2} J.E. Andersen, J. Mattes \& N. Reshetikhin,
  \emph{Quantization of the Algebra of Chord
    Diagrams}. Math. Proc. Camb. Phil. Soc. {\bf 124} pp.451--467
  (1998).



\bibitem[AU1]{AU1} J.E. Andersen \& K. Ueno, \emph{Abelian Conformal
    Field theories and Determinant Bundles}, International Journal of
  Mathematics.  {\bf 18}, (2007) 919--993.

\bibitem[AU2]{AU2} J.E. Andersen \& K. Ueno, \emph{Constructing
    modular functors from conformal field theories}, Journal of Knot
  theory and its Ramifications. {\bf 16} 2 (2007), 127--202.

\bibitem[AU3]{AU3} J.E. Andersen \& K. Ueno, \emph{Modular functors
    are determined by their genus zero data}, math.QA/0611087.

\bibitem[AU4]{AU4} J.E. Andersen \& K. Ueno, \emph{Construction of the
    Reshetikhin-Turaev TQFT from conformal field theory}, In
  preparation.

\bibitem[AV1]{AV1} J.E. Andersen \& R. Villemoes, \emph{Degree One
    Cohomology with Twisted Coefficients of the Mapping Class Group},
  arXiv:0710.2203v1.

\bibitem[AV2]{AV2} J.E. Andersen \& R. Villemoes, \emph{The first
    cohomology of the mapping class group with coefficients in
    algebraic functions on the SL(2, C) moduli space},
  arXiv:0802.4372v1.

\bibitem[AV3]{AV3} J.E. Andersen \& R. Villemoes, \emph{Cohomology of
    mapping class groups and the abelian moduli space}, In
  preparation.


\bibitem[At]{At} M. Atiyah, \emph{The Jones-Witten invariants of
    knots}. S\'{e}minaire Bourbaki, Vol. 1989/90. Ast\'{e}risque No.
  {\bf 189-190} (1990), Exp. No. 715, 7--16.

\bibitem[AB]{AB} M. Atiyah \& R. Bott, \emph{The Yang-Mills equations
    over Riemann surfaces}.  Phil. Trans. R. Soc. Lond., Vol.  {\bf
    A308} (1982) 523--615.


\bibitem[ADW]{ADW} S.~Axelrod, S.~Della~Pietra, E.~Witten,
  \emph{Geometric quantization of Chern Simons gauge theory},
  J.Diff.Geom. {\bf 33} (1991) 787--902.









\bibitem[BK]{BK} B. Bakalov and A. Kirillov, \emph{Lectures on tensor
    categories and modular functors}, AMS University Lecture Series,
  {\bf 21} (2000).

\bibitem[BHV]{BHV} B. Bekka, P. de la Harpe \& A. Valette,
  \emph{Kazhdan's Proporty (T)}, In Press, Cambridge University Press
  (2007).

\bibitem[Besse]{Besse} A. L. Besse, \emph{Einstein Manifolds},
  Springer-Verlag, Berlin (1987).



\bibitem[B1]{B1}C. Blanchet, \emph{Hecke algebras, modular categories
    and $3$-manifolds quantum invariants}, Topology {\bf 39} (2000),
  no. 1, 193--223.

\bibitem[BHMV1]{BHMV1} C. Blanchet, N. Habegger, G. Masbaum \&
  P. Vogel, \emph{Three-manifold invariants derived from the Kauffman
    Bracket}.  Topology {\bf 31} (1992), 685--699.


\bibitem[BHMV2]{BHMV2} C. Blanchet, N. Habegger, G. Masbaum \&
  P. Vogel, \emph{Topological Quantum Field Theories derived from the
    Kauffman bracket}. Topology {\bf 34} (1995), 883--927.


\bibitem[BC]{BC} S. Bleiler \& A. Casson, \emph{Automorphisms of
    sufaces after Nielsen and Thurston}, Cambridge University Press,
  1988.

\bibitem[BMS]{BMS} M. Bordeman, E. Meinrenken \& M.  Schlichenmaier,
  \emph{Toeplitz quantization of K{\"a}hler manifolds and $gl(N), N
    \ra \infty$ limit}, Comm. Math. Phys. {\bf 165} (1994), 281--296.

\bibitem[BdMG]{BdMG} L. Boutet de Monvel \& V. Guillemin, \emph{The
    spectral theory of Toeplitz operators}, Annals of Math. Studies
  {\bf 99}, Princeton University Press, Princeton.

\bibitem[BdMS]{BdMS} L. Boutet de Monvel \& J. Sj\"{o}strand,
  \emph{Sur la singularit\'{e} des noyaux de Bergmann et de
    Szeg\"{o}}, Asterique {\bf 34-35} (1976), 123--164.


\bibitem[DN]{DN} J.-M. Drezet \& M.S. Narasimhan, \emph{Groupe de
    Picard des vari\'{e}t\'{e}s de modules de fibr\'{e}s semi-stables
    sur les courbes alg\'{e}briques}, Invent. math. {\bf 97} (1989)
  53--94.



\bibitem[Fal]{Fal} G.~Faltings, \emph{Stable G-bundles and projective
    connections}, J.Alg.Geom. {\bf 2} (1993) 507--568.

\bibitem[FLP]{FLP} A. Fathi, F. Laudenbach \& V. Po\'{e}naru,
  \emph{Travaux de Thurston sur les surfaces}, Ast\'{e}risque {\bf
    66--67} (1991/1979).

\bibitem[Fi]{Fi} M. Finkelberg, \emph{An equivalence of fusion
    categories}, Geom. Funct. Anal. {\bf 6} (1996), 249--267.

\bibitem[Fr]{Fr} D.S. Freed, \emph{Classical Chern-Simons Theory, Part
    1}, Adv. Math. {\bf 113} (1995), 237--303.


\bibitem[FWW]{FWW} M. H. Freedman, K. Walker \& Z. Wang, \emph{Quantum
    $\SU(2)$ faithfully detects mapping class groups modulo center}.
  Geom. Topol. {\bf 6} (2002), 523--539

\bibitem[FR1]{FR1} V. V. Fock \& A. A Rosly, \emph{Flat connections
    and polyubles}. Teoret.  Mat. Fiz. {\bf 95} (1993), no. 2,
  228--238; translation in Theoret. and Math. Phys. {\bf 95} (1993),
  no. 2, 526--534

\bibitem[FR2]{FR2} V. V. Fock \& A. A Rosly, \emph{Moduli space of
    flat connections as a Poisson manifold}.  Advances in quantum
  field theory and statistical mechanics: 2nd Italian-Russian
  collaboration (Como, 1996). Internat. J. Modern Phys. B {\bf 11}
  (1997), no. 26-27, 3195--3206.

\bibitem[vGdJ]{vGdJ} B. Van Geemen \& A. J. De Jong, \emph{On
    Hitchin's connection}, J. of Amer. Math. Soc., {\bf 11} (1998),
  189--228.


\bibitem[Go1]{Go1} W. M. Goldman, \emph{Ergodic theory on moduli
    spaces}, Ann. of Math. (2) 146 (1997), no. 3, 475--507.


\bibitem[Go2]{Go2} W. M. Goldman, \emph{Invariant functions on Lie
    groups and Hamiltonian flows of surface group representations},
  Invent. Math. 85 (1986), no. 2, 263--302.


\bibitem[GR]{GR} S. Gutt \& J. Rawnsley, \emph{Equivalence of star
    products on a symplectic manifold}, J. of Geom. Phys., {\bf 29}
  (1999), 347--392.


\bibitem[H]{H} N.~Hitchin, \emph{Flat connections and geometric
    quantization}, Comm.Math.Phys., {\bf 131} (1990) 347--380.




\bibitem[KS]{KS} A. V. Karabegov \& M.  Schlichenmaier,
  \emph{Identification of Berezin-Toeplitz deformation quantization},
  J. Reine Angew. Math. {\bf 540} (2001), 49--76.

\bibitem[Kar]{Kar} A. V. Karabegov, \emph{Deformation Quantization
    with Separation of Variables on a K\"ahler Manifold},
  Comm. Math. Phys. {\bf 180} (1996) (3), 745---755.

\bibitem[Kac]{Kac} V. G. Kac, \emph{Infinite dimensional Lie
    algebras}, Third Edition, Cambridge University Press, (1995).

\bibitem[Kazh]{Kazh} D. Kazhdan, \emph{Connection of the dual space of
    a group with the structure of its closed subgroups},
  Funct. Anal. Appli. {\bf 1} (1967), 64--65.

\bibitem[KL]{KL} D. Kazhdan \& G. Lusztig, \emph{Tensor structures
    arising from affine Lie algebras I}, J. AMS, {\bf 6} (1993),
  905--947; II J. AMS, {\bf 6} (1993), 949--1011; III J. AMS, {\bf 7}
  (1994), 335--381; IV, J. AMS, {\bf 7} (1994), 383--453.



\bibitem[La1]{La1} Y. Laszlo, \emph{Hitchin's and WZW connections are
    the same}, J. Diff. Geom. {\bf 49} (1998), no. 3, 547--576.



\bibitem[M1]{M} G. Masbaum, \emph{An element of infinite order in
    TQFT-representations of mapping class groups}. Low-dimensional
  topology (Funchal, 1998), 137--139, Contemp. Math., {\bf 233}, Amer.
  Math. Soc., Providence, RI, 1999.

\bibitem[M2]{M2}G. Masbaum. \emph{Quantum representations of mapping
    class groups}. In: Groupes et G\'{e}om\'{e}trie (Journ\'{e}e
  annuelle 2003 de la SMF). pages 19--36


\bibitem[MS]{MS} G. Moore and N. Seiberg, \emph{Classical and quantum
    conformal field theory}, Comm. Math. Phys. {\bf 123} (1989),
  177--254.

\bibitem[NS1]{NS1} M.S. Narasimhan and C.S. Seshadri,
  \emph{Holomorphic vector bundles on a compact Riemann surface},
  Math. Ann. {\bf 155} (1964) 69--80.


\bibitem[NS2]{NS2} M.S. Narasimhan and C.S. Seshadri, \emph{Stable and
    unitary vector bundles on a compact Riemann surface}, Ann. Math.
  {\bf 82} (1965) 540--67.



\bibitem[R1]{R1} T.R. Ramadas, \emph{Chern-Simons gauge theory and
    projectively flat vector bundles on $M_g$}, Comm. Math. Phys. {\bf
    128} (1990), no. 2, 421--426.


\bibitem[RSW]{RSW} T.R. Ramadas, I.M. Singer and J. Weitsman,
  \emph{Some Comments on Chern -- Simons Gauge Theory},
  Comm. Math. Phys. {\bf 126} (1989) 409-420.


\bibitem[RT1]{RT1} N. Reshetikhin \& V. Turaev, \emph{Ribbon graphs
    and their invariants derived fron quantum groups},
  Comm. Math. Phys.  {\bf 127} (1990), 1--26.

\bibitem[RT2]{RT2} N. Reshetikhin \& V. Turaev, \emph{Invariants of
    $3$-manifolds via link polynomials and quantum groups}, Invent.
  Math. {\bf 103} (1991), 547--597.

\bibitem[Ro]{Ro} J. Roberts, \emph{Irreducibility of some quantum
    representations of mapping class groups}. J. Knot Theory and its
  Ramifications 10 (2001) 763 -- 767.



\bibitem[Sch]{Sch} M. Schlichenmaier, \emph{Berezin-Toeplitz
    quantization and conformal field theory}, Thesis.

\bibitem[Sch1]{Sch1} M. Schlichenmaier, \emph{Deformation quantization
    of compact K\"{a}hler manifolds by Berezin-Toeplitz quantization}.
  In {\em Conf\'{e}rence Mosh\'{e} Flato 1999, Vol. II (Dijon)},
  289--306, Math. Phys. Stud., {\bf 22}, Kluwer Acad. Publ.,
  Dordrecht, (2000), 289--306.


\bibitem[Sch2]{Sch2} M. Schlichenmaier, \emph{Berezin-Toeplitz
    quantization and Berezin transform}.  In {\em Long time behaviour
    of classical and quantum systems (Bologna, 1999)}, Ser. Concr.
  Appl. Math., {\bf 1}, World Sci. Publishing, River Edge, NJ, (2001),
  271--287.

\bibitem[Se]{Segal}G. Segal, \emph{The Definition of Conformal Field
    Theory}, Oxford University Preprint (1992).

\bibitem[Th]{Th} W. Thurston, \emph{On the geometry and dynamics of
    diffeomorphisms of surfaces}, Bull of Amer. Math. Soc. 19 (1988),
  417--431.

\bibitem[TUY]{TUY} A. Tsuchiya, K. Ueno \& Y. Yamada, \emph{Conformal
    Field Theory on Universal Family of Stable Curves with Gauge
    Symmetries}, {\em Advanced Studies in Pure Mathmatics}, {\bf 19}
  (1989), 459--566.

\bibitem[T]{T} V. G. Turaev, \emph{Quantum invariants of knots and
    3-manifolds}, de Gruyter Studies in Mathematics, 18.  Walter de
  Gruyter \& Co., Berlin, 1994. x+588 pp. ISBN: 3-11-013704-6

\bibitem[Tuyn]{Tuyn} Tuynman, G.M., \emph{Quantization: Towards a
    comparision between methods}, J. Math. Phys. {\bf 28} (1987),
  2829--2840.

\bibitem[Vi]{Vi} R. Villemoes, \emph{The mapping class group orbit of
    a multicurve}, arXiv:0802.3000v2


\bibitem[Wa]{Walker} K. Walker, \emph{On Witten's 3-manifold
    invariants}, {\em Preliminary version \# 2, Preprint} 1991.


\bibitem[Wi]{W1} E. Witten, \emph{Quantum field theory and the Jones
    polynomial}, Commun. Math. Phys {\bf 121} (1989) 351--98.

\bibitem[Wo]{Woodhouse} N.J. Woodhouse, \emph{Geometric Quantization},
  Oxford University Press, Oxford (1992).


\end{thebibliography}

\end{document}